\documentclass[10pt]{article}
\usepackage{amsmath}
\usepackage{amsthm}
\usepackage{amsfonts}
\usepackage{amssymb}
\usepackage{amscd}
\usepackage[all]{xy}
\usepackage[percent]{overpic}
\usepackage{graphicx}
\usepackage{xcolor}
\usepackage{enumerate}

\font\smallsc=cmcsc10
\font\smallsl=cmsl10

\newtheorem{theorem}{Theorem}

\newtheorem{corollary}[theorem]{Corollary}

\newtheorem{lemma}[theorem]{Lemma}

\newtheorem{proposition}[theorem]{Proposition}

\newcommand{\Ccal}{\mathcal C}
\newcommand{\Pbb}{\mathbb P}
\newcommand{\Ocal}{\mathcal O}

\renewcommand{\ni}{n^{(i)}}
\newcommand{\mi}{m^{(i)}}
\newcommand{\n}[1]{n^{(#1)}}
\newcommand{\nlinha}[1]{{n'}^{(#1)}}

\begin{document}

\title{On the gonality of stable curves}
\author{Juliana Coelho\\{\scriptsize julianacoelhochaves@id.uff.br}
\and Frederico Sercio\\{\scriptsize fred.feitosa@ufjf.edu.br}}

\maketitle

\begin{abstract}
In this paper we use admissible covers to investigate the gonality of a stable   curve $C$
over $\mathbb{C}$.
If $C$ is irreducible, we compare its gonality to that of its normalization.
If $C$ is reducible, we compare its gonality to that of its irreducible components. In both cases we obtain lower and upper
bounds. 
Furthermore, we show that four  admissible covers constructed give rise to generically injective maps between Hurwitz schemes. 
We show that the closures of the images of  three of these maps are   components of the boundary of the target Hurwitz schemes, and  the closure of the image of the remaining map is a  component of a certain codimension-1 subscheme of the boundary of the target Hurwitz scheme.
\end{abstract}

\section{Introduction}

The gonality is an important numerical invariant in the study of algebraic curves and is related to the  study of finite maps between curves.  A smooth curve $C$ is said to be $k$-gonal if it admits a degree-$k$ map 
$C\rightarrow \mathbb{P}^{1}$. 
If $C$ has genus $g$,  it follows from the Riemann-Roch theorem  that $C$ is $k$-gonal for any $k\geq g+1$. Furthermore, it follows from Brill-Noether theory 
that $C$ is $k$-gonal for some $k\leq (g+3)/2$.

In 1891, A. Hurwitz  gave in \cite{Hu} a complex structure to the set $H_{k,g}$ of all $k$-sheeted simple coverings of $\mathbb{P}^{1}$ by smooth complex curves of genus $g$ with $b=2g+2k-2$ branch  points. Using calculations of L\"{u}roth and Clebsch, he proved that $H_{k,g}$ is connected.  In 1921, F. Severi  proved in \cite{S} the irreducibility of the (coarse) moduli space ${M}_{g}$ of smooth curves of genus $g$ by combining Hurwitz's result with the fact that every smooth curve of genus $g$ appears as a $k$-sheeted simple covering of $\mathbb{P}^{1}$, if $k\geq g+1$.   Later on, in 1969, W. Fulton considered in  \cite{F}  the problem of constructing $H_{k,g}$ over the integers. He defined  a finite \'{e}tale morphism  from $H_{k,g}$ onto an open subscheme of the quasi-projective scheme parametrizing $b$-tuples of distinct points of $\mathbb{P}^{1}$,  by mapping a covering to its branch locus.  

When studying smooth curves in families, one is compelled to consider singular curves as well. Indeed, the moduli space $ M_g$ of smooth curves of genus $g$ is not compact, and a compactification is given by the moduli space $\overline{ M}_g$ of  stable curves of (arithmetic) genus $g$. A stable curve $C$  is said to be $k$-gonal if it is a limit of smooth $k$-gonal curves. In 1982, Harris and Mumford   characterized $k$-gonal stable curves in terms of  $k$-sheeted admissible covers in \cite{HM82}. This result is a natural connection between the study of gonality of stable curves and the study of (generalized) Hurwitz schemes parametrizing admissible covers.

Not much is known about the gonality of stable curves.  
Ballico investigated the gonality of graph curves in \cite{B1}.
Caporaso showed in \cite{C} a relation between the gonality of a stable curve and that of its graph.
In her thesis \cite{Br}, Brannetti studied the gonality of stable curves with two components.

\subsection{Main Results}

In this work we use admissible covers to investigate the relationship between the gonality of a complex stable
curve  and that  of its partial normalizations and  irreducible components. 
A $k$-sheeted admissible cover for a nodal curve $C$ is   a finite morphism  $\pi\colon C\rightarrow B$ of degree $k$ satisfying a few conditions, where $B$ is a nodal curve of genus 0  (see Section \ref{sec:gon} or \cite{HM82} for the precise definition). 
By \cite{HM82}, a stable curve $C$ is $k$-gonal if and only if there is a $k$-sheeted admissible cover for a nodal curve $C'$ stably equivalent to $C$, that is, such that $C$ is obtained from $C'$ by contracting to a node every rational component of $C'$ that meets the rest of the curve in only one or two points. It is a well known fact that a $k$-gonal stable curve is also $r$-gonal for every $r>k$. We recover this fact in Theorem \ref{prop:step2} by explicitly constructing a $(k+1)$-sheeted admissible cover from a $k$-sheeted one.

Now let $n$ be a node of $C$ and let $C_n$ be the normalization of $C$ at $n$. We say that $n$ is a separating node if $C_n$ is disconnected. If $n$ is non-separating and $C_n$ is $k$-gonal then there is a $k$-sheeted admissible cover $\pi\colon C_n'\rightarrow B$, where $C_n'$ is stably equivalent to $C_n$. 
In Theorem \ref{irred:lem} we construct from $\pi$ an admissible cover for a curve stably equivalent to $C$, thus comparing the gonality of $C$ to that of $C_n$. The construction is different depending on the behaviour of $\pi$ on the points of $C_n'$ over the node $n$ and the resulting admissible cover can be either $k$-sheeted or $(k+1)$-sheeted. As a consequence we obtain in Corollary \ref{irred:cor} a comparison between the gonality of an irreducible nodal curve $C$ and that of its normalization.

Now, if $n$ is a separating node, then $C_n$ can be seen as a disjoint union of two subcurves of $C$. More generally, let $Y_1,\ldots,Y_r$ be connected subcurves of $C$ such that $C=Y_1\cup\dots\cup Y_r$ and $Y_i\cap Y_{i'}$ is either empty or finite for $i\neq i'$. If each $Y_i$ is $k_i$-gonal, then for each $i$ there is a  $k_i$-sheeted admissible cover $\pi_i\colon Y_i'\rightarrow B_i$, where $Y_i'$ is stably equivalent to $Y_i$. 
Under certain conditions on the maps $\pi_1,\ldots,\pi_r$, we produce a $k$-sheeted  admissible cover for a curve stably equivalent to $C$, where $k=k_1+\ldots+k_r-(\delta_1+\ldots+\delta_r)/2$ and $\delta_i$ is the number of points where $Y_i$ meets the rest of the curve $C$ (see Theorem \ref{red:nonconjug0} for the precise statement). This allows us to compare in Corollary \ref{cor:bound} the gonality of $C$ to that of its irreducible components.

The theorem below summarizes the bounds obtained in Section \ref{sec:gon}: 
\smallskip
\pagebreak

\noindent{\bf Theorem A} \; {\it Let $C$ be a stable curve.

\noindent (i) If $C$ is irreducible with $\delta$ nodes and its  normalization  is $\tilde k$-gonal, then $C$ is $k$-gonal for some $k$ satisfying 
\[\tilde k\leq k\leq \tilde k+\delta.\]

\noindent (ii) If $C$ is reducible with irreducible components $C_{1},\ldots,C_{p}$ and each $C_i$ is $k_i$-gonal, then  $C$ is $k$-gonal for some $k$ satisfying 
\[ k_{1}+\ldots+k_{p}-\delta\leq k\leq k_{1}+\ldots+k_{p}+\delta-2(p-1),\]
where  $\delta$ is the number of external nodes of $C$.
}\smallskip

In Section \ref{sec:hurwitzscheme}  we view  four constructions of admissible covers done in Section \ref{sec:gon} as rational maps  between pointed Hurwitz schemes.
A pointed Hurwitz scheme $\overline H_{k,g,\ell_1,\ldots,\ell_n}$ parametrizes $k$-sheeted admissible covers for genus-$g$ curves $C$ together with $\ell_1+\ldots+\ell_n$ marked points on $C$ satisfying certain conditions (see Section \ref{sec:hurwitzscheme} for the precise definition). 
The four maps are shown to be generically injective in Theorem \ref{thm:geninj} and, as a consequence, we can determine their images.

It's easy to describe the boundary of the Hurwitz scheme $\overline H_{k,g}$ parametrizing $k$-sheeted admissible covers $\pi\colon C\rightarrow B$ where $C$ is a nodal curve of genus $g$. 
For instance, the closure   $\Delta^n$ of the locus of those admissible covers such that $B$ has $n+1$ components is in the boundary. Moreover, $\Delta^n$ has pure codimension $n$ in  $\overline H_{k,g}$ but, in general, it is not irreducible. 
We show in Theorem \ref{thm:images} that the closures of the images of the maps of Hurwitz schemes we constructed are irreducible components of $\Delta^n$, for some $n=1,2$.

The results of Section \ref{sec:hurwitzscheme}  are summarized in the following theorem:\smallskip

\noindent{\bf Theorem B} \; {\it There are generically injective rational maps 
\[\gamma\colon \overline{H}_{k,g,1}\dashrightarrow\overline{H}_{k+1,g},
\quad\quad
\phi\colon \overline{H}_{k,g,1,1}\dashrightarrow\overline{H}_{k+1,g+1},
\quad\quad
\psi\colon \overline{H}_{k,g,2}\dashrightarrow\overline{H}_{k,g+1}
\]
and
\[\lambda\colon \overline{H}_{k_{1},g_{1},\delta}
\times
\overline{H}_{k_{2},g_{2},\delta} \dashrightarrow \overline{H}_{k_{1}+k_{2}-\delta,g_{1}+g_{2}+\delta-1}\]
such that 
the closures of the images of $\gamma$,  $\psi$ and $\lambda$ are irreducible components of $\Delta^{1}$, and
the closure of the image of  $\phi$ is an irreducible component of $\Delta^{2}$, 
in the corresponding Hurwitz scheme. 
}\smallskip

\section{Technical background}

In this paper we always work over the field of complex numbers $\mathbb{C}$.
A \textit{curve} $C$ is a connected,
projective and reduced scheme of dimension $1$ over $\mathbb{C}$. 
The \emph{genus} of $C$ is $g(C):=h^1(C,\Ocal_C)$.
A \textit{subcurve} $Y $ of $C$ is a reduced subscheme of pure dimension $1$, or equivalently, a reduced union of irreducible components of $C$. 
If $Y\subseteq C$ is a subcurve, we set $Y^{c}:=\overline{C\smallsetminus Y}$. 

A \textit{nodal curve} $C$ is a curve with at most ordinary double points. A node $n$ of  $C$ is said to be  \textit{external} if $n\in Y\cap Y^{c}$ for some subcurve $Y$ of $C$, otherwise the node is called \textit{internal}. 
An external node is said to be \emph{separating} if there is a subcurve $Y$ of $C$ such that $Y\cap Y^c=\{n\}$. In this case, the subcurves $Y$ and $Y^c$ are said to be \emph{tails} of $C$ associated to the separating node $n$.

Let $C$ be a nodal curve and let $\nu:\widetilde{C}\rightarrow C$ be its
normalization. Let $n$ be a node of $C$. The \textit{branches over the node} $n$ are the points $\n{1} ,\n{2} \in \widetilde{C}$ such that $\nu(\n1)=\nu(\n2)=n$.

A \textit{family of curves} is a proper and flat morphism $f:\mathcal{C}\rightarrow S$ whose fibers are curves. If $s\in S$, we denote by ${C}_{s}:=f^{-1}(s)$ its fiber over $s$. 
A \textit{smoothing} of a curve $C$ is a 
family $f:\mathcal{C}\rightarrow S$, where $S=\text{Spec}(\mathbb{C}[[t]])$, such that the  generic fiber of $f$ is smooth and the
special fiber is isomorphic to $C$, henceforth identified with $C$.

Let $S$ be a scheme and let $g$ and $n$ be non-negative integers such that
$2g-2+n>0$. A \textit{$n$-pointed stable curve of genus $g$ over $S$} is a
family of curves $f:\mathcal{C}\rightarrow S$ together with $n$
distinct sections $\sigma_{i}:S\rightarrow \Ccal$ such that, for every $s\in S$:
\begin{enumerate}
\item The geometric fiber $C_{s}$ of $f$ is a genus-$g$ nodal curve and $\sigma_1(s),\ldots,\sigma_n(s)$ are distinct  smooth points of $C_s$;
\item For every  smooth rational subcurve $E$ of $C_s$, 
the number of points where  $E$ meets $E^{c}$ plus the number of indices $i$ such that  $\sigma_{i}(s)$ lies on $E$ is at least three.
\end{enumerate}

A \textit{stable curve over $S$} is a 0-pointed stable curve over $S$. A \emph{stable curve} is a stable curve over Spec$(\mathbb{C})$.
The scheme $\overline M_{g,n}$  parametrizing $n$-pointed stable curves of genus $g$  is   projective and irreducible, by  \cite{K2} and \cite{K3}. Set  $\overline M_g=\overline M_{g,0}$.

\section{Admissible covers}\label{sec:gon}

A smooth curve is \emph{$k$-gonal} if it admits a $\mathfrak{g}_{k}^{1}$, that is, a line bundle of degree $k$ having at least two independent sections.
Equivalently, a smooth curve is $k$-gonal if it admits a map of degree  $k$ or less  to $\mathbb{P}^{1}$. Now, a stable curve $C$ is 
 \emph{$k$-gonal} if it is a limit of smooth $k$-gonal curves in $\overline M_g$. More precisely, $C$ is $k$-gonal if it admits a
 smoothing $f:\mathcal{C}\rightarrow S$ whose geometric general fiber is
a $k$-gonal smooth curve and the special fiber is $C $.

Let $C$ and $C'$ be nodal curves. We say that  $C^{\prime}$ is \emph{stably equivalent} to $C$ if $C$ can be obtained from $C^{\prime}$ by contracting to a
point some of the smooth rational components of $C^{\prime}$ meeting the other
components of $C^{\prime}$ in only one or two points. 
We remark that,  although a $k$-gonal stable curve $C$ may not admit a $\mathfrak{g}_{k}^{1}$, there is always 
a nodal curve $C^{\prime}$  stably equivalent to $C$
admiting a $\mathfrak{g}_{k}^{1}$. 
Indeed, if $C$ is $k$-gonal then there exists a
 smoothing $f:\mathcal{C}\rightarrow S$ whose geometric general fiber is
a $k$-gonal smooth curve and the special fiber is $C$.
Changing the base and blowing up  $\mathcal C$ if necessary, we obtain a  regular family $f':\mathcal C'\rightarrow S'$ whose generic fiber is isomorphic to that of $f$ and whose special fiber is a nodal curve $C'$ stably equivalent to $C$.
Since $\mathcal C'$ is regular, there exists a line bundle $\mathcal L$ on $\mathcal C'$ whose restriction 
to the generic fiber of $f'$ is a $\mathfrak g_k^1$. Then by uppersemicontinuity, the restriction of $\mathcal L$ to the special fiber $C'$ is also a $\mathfrak g_k^1$.
However, the  converse is not always true, and a stable curve admiting a  $\mathfrak{g}^{1}_{k}$ may fail to be $k$-gonal. Indeed, 
if $C$ is the curve obtained by gluing a Weierstrass point of a hyperelliptic smooth curve to a non-Weierstrass point, then the resulting $\mathfrak{g}^1_2$ on $C$ is not a limit  of $\mathfrak{g}^{1}_{2}$ on smooth curves. 

%Note that  a $k$-gonal stable curve $C$ does admit a $\mathfrak{g}_{k}^{1}$ since,  after a base extension if necessary, we can assume that $\mathcal C$ is regular and that there exists a line bundle $\mathcal L$ on $\mathcal C$ whose restriction to the generic fiber of $f$ is a $\mathfrak g_k^1$. Then by uppersemicontinuity, the restriction of $\mathcal L$ to the special fiber $C$ is also a $\mathfrak g_k^1$. However, the  converse is not always true. Indeed,  if $C$ is the curve obtained by gluing a Weierstrass point of a hyperelliptic smooth curve to a non-Weierstrass point, then the resulting $\mathfrak{g}^1_2$ on $C$ is not a limit  of $\mathfrak{g}^{1}_{2}$ on smooth curves. 

Alternatively, $k$-gonal stable curves can be characterized in terms of admissible covers.
A \emph{$k$-sheeted admissible cover} consists of  a finite morphism $\pi:C\rightarrow B$ of degree $k$, such that $B$ and $C$ are nodal curves and:
\begin{enumerate}
\item $\pi^{-1}\left(B^{\text{sing}}\right)  =C^{\text{sing}}$;
\item $\pi$ is simply branched away from $C^{\text{sing}}$, that is, 
over each smooth point of $B$ there exists at most one point of $C$ 
where $\pi$ is ramified and this point has ramification index 2;
\item $B$ is a stable pointed curve of genus 0, when considered with its smooth points on the branch locus of $\pi$;
\item for every node $q$ of $B$ and every node $n$ of $C$ 
lying over it, the
two branches of $C$ over $n$ map to the branches of $B$ over $q$ with the same
ramification index.
\end{enumerate}
It follows from the Riemann-Hurwitz theorem that the number of smooth branch points of $\pi$ is  $b=2g+2k-2$, where $g$ is the genus of $C$.

Let $\pi\colon C\rightarrow B$ be a $k$-sheeted admissible cover. Let   $n$ be a node of $C$ and set $q=\pi(n)$. Condition 4 above implies that, locally around $n$, the curve $C$ can be described as $xy=t$  and, locally around $q$, the curve  $B$ can be described as $uv=t^\ell$ for some $\ell$. Moreover, the map $\pi$ is given by $u=x^\ell$ and $v=y^\ell$.

 The following result is a consequence of \cite[Thm. 4, p. 58]{HM82}
and relates the notion of admissible cover to that of gonality of stable curves.

\begin{theorem}[Harris-Mumford]\label{HarrMum}
A stable curve $C$ is $k$-gonal if and only if
there exists a $k$-sheeted admissible cover $C^{\prime}\rightarrow B$, 
where $C^{\prime}$ is stably equivalent to $C$.
\end{theorem}
\begin{proof}
Cf. \cite[Thm. 3.160, p. 185]{HM}.
\end{proof}

The following lemma relates the irreducible components of a nodal curve $C$ to those of a curve stably equivalent to it.

\begin{lemma}\label{lem:admiss}
Let $C$ be a nodal curve and let $\pi:C'\rightarrow B$  be an admissible cover, where $C'$ is stably equivalent to $C$. Denote by $\tau:C'\rightarrow C$ the induced map. Then  for each irreducible component $C_r$ of $C$ there is an irreducible component $C_r'$ of $C'$ such that $\tau(C_r')=C_r$ and $\tau|_{C_r'}$ is a normalization map.
\end{lemma}
\begin{proof}
First we note that, since the components of $B$ are smooth and $\pi^{-1}(B^{\text{sing}})=C'^{\text{sing}}$, the components of $C'$ are also smooth.
Now fix a component $C_r$ of $C$. By definition, there is an unique component $C_r'$  of $\tau^{-1}(C_r)$ dominating $C_r$. Since $C_r'$ is smooth, $\tau|_{C_r'}:C_r'\rightarrow C_r$ is a normalization map.
\end{proof}

Let $C$ be a nodal curve and $C'$ a nodal curve stably equivalent to $C$ with  induced map $\tau:C'\rightarrow C$. 
We say that a point $p'\in C'$  \emph{lies over} a point $p\in C$ if  $\tau(p')=p$.

The following result shows in particular that, if a stable curve is $k$-gonal for some $k$, then it is $r$-gonal for every $r>k$.

\begin{theorem}\label{prop:step2}
Let $\pi:C\rightarrow B$ be a degree-$k$ finite morphism of nodal curves. 
\begin{enumerate}[(a)]
\item If $\pi$ satisfies  conditions (1), (3) and (4) of an  admissible cover, then there  is a $k$-sheeted admissible cover $\pi':C'\rightarrow B'$ such that $C'$ (resp. $B'$) is stably equivalent to $C$ (resp. $B$) , contains $C$ (resp. $B$) as a subcurve and $\pi'|_{C}=\pi$.
\item There is a degree-$(k+1)$  finite morphism of nodal curves  $\pi':C'\rightarrow B'$ such that $C'$ (resp. $B'$) is stably equivalent to $C$ (resp. $B$), contains $C$ (resp. $B$) as a subcurve and $\pi'|_{C}=\pi$. Moreover, if $\pi$ is an admissible cover, then so is $\pi'$. 
\end{enumerate}
\end{theorem}
\begin{proof}
(a) For each  smooth point $q\in B$ such that
\[\pi^{-1}(q)  =\rho_{1}m_{1}+\ldots+\rho_{\ell}m_{\ell}\]
where $m_{1},\ldots,m_{\ell}\in C$ and 
 $\rho_{j}\geq3$ for some
$j=1,\ldots,\ell$ or $\rho_{j}=\rho_{j'}=2$ 
for some
$j,j'=1,\ldots,\ell$ with $j'\neq j$, we glue (see Figure \ref{fig:step2a}):

\begin{figure}[h!]
\centering
\includegraphics[height=2.6in]{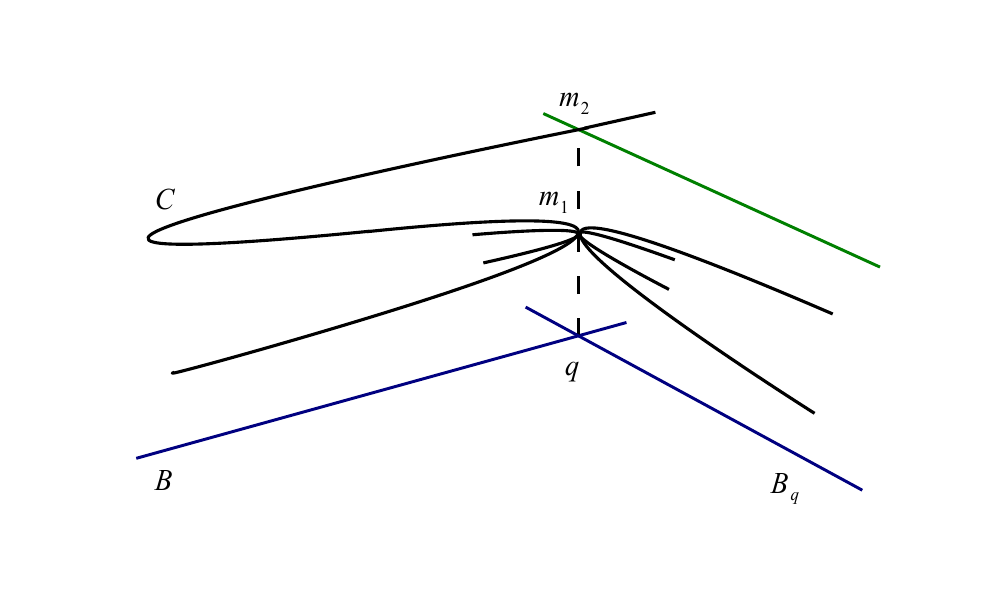}
\caption{Theorem \ref{prop:step2} (a) for $k=5$, $\rho_1=4$ and $\rho_2=1$}
\label{fig:step2a}
\end{figure}

\begin{itemize}
\item a copy of $\Pbb^1$  at $q$, denoted by $B_q$;

\item a copy of $\mathbb{P}^{1}$ at every $m_{j}$, denoted by $L_j$, mapping to $B_q$ via a degree-$\rho_{j}$ map totally ramified at $m_j$, simply ramified away from $m_j$ and unramified over the branch points of $L_{j'}\rightarrow B_q$, for $j,j'=1,\ldots,\ell$ with $j'\neq j$.
\end{itemize}

At the end of this proccess we obtain from $C$ a nodal curve $C^{\prime}$ stably equivalent to $C$,  from $B$ a nodal curve $B^{\prime}$ stably equivalent to
$B$, and a degree-$k$ map $\pi':C'\rightarrow B'$ given by $\pi$ when restricted
to $C$, and by the maps described above when restricted to the added rational
components of $C^{\prime}$. 
By construction, it is clear that $\pi'$ satisfies conditions (1), (2) and (4) of an admissible cover.
It also satisfies condition (3) since, by the Riemann-Hurwitz  theorem, the restriction of $\pi'$ to $L_j$ has $2\rho_j-2$ ramification points, counted with multiplicity. 
By construction, the multiplicity of $q$ is $\rho_j-1$ and there must be  $\rho_j-1$ other simple ramification points.
Hence, if $\rho_{j}\geq3$ then $B_q$ has at least two marked points distinct from $q$. 
If $\rho_{j}=\rho_{j'}=2$ then the ramification points  of the maps  from $L_j$ and $L_{j'}$ to $ B_q$
 give three distinct marked points in $B_q$, one of which is $q$. Hence $B'$ is stable when considered with the smooth branch locus of $\pi'$, showing that $\pi'$ is an admissible cover.

(b) Choose any point $m$ in the smooth locus of $C$, let $q=\pi(m)$ and
\[
\pi^{-1}(q)  =\rho_{0}m_{0}+ \rho_{1}m_{1}+\ldots+\rho_{\ell}m_{\ell}
\]
where  $m_0=m$ and $m_{1},\ldots,m_{\ell}\in C$. We glue (see Figure \ref{fig:step2c}):

\begin{figure}[h!]
\centering
\includegraphics[height=2.6in]{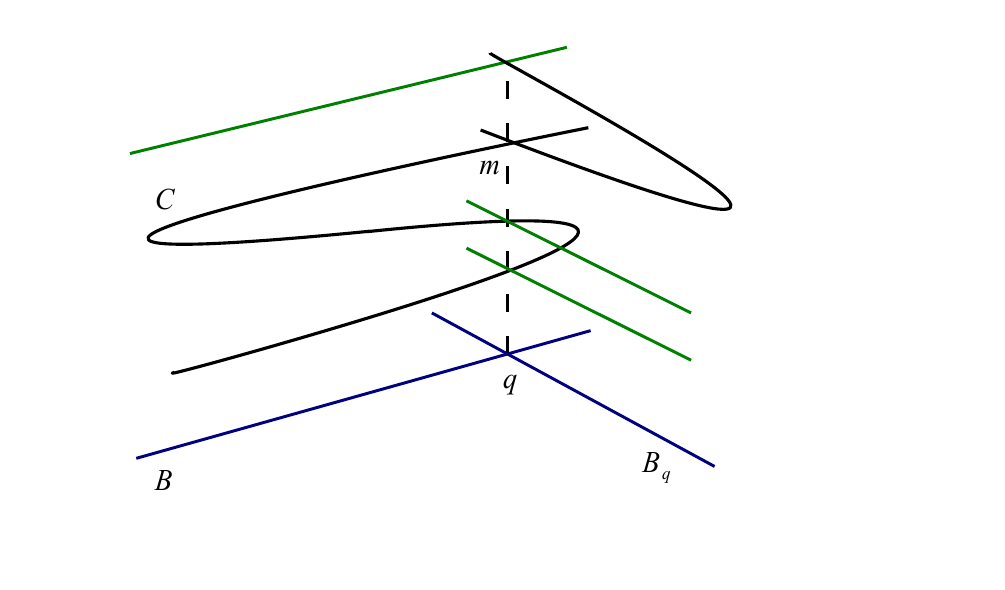}
\caption{Theorem \ref{prop:step2} (b) for $k=3$}
\label{fig:step2c}
\end{figure}

\begin{itemize}
\item a copy of $\Pbb^1$  at $q$, denoted by $B_q$;

\item a copy of $\mathbb{P}^{1}$ at every $m_{j}$, denoted by $L_j$, mapping to $B_q$ via a degree-$\rho_{j}$ map totally ramified at $m_j$, simply ramified away from $m_j$
and unramified over the branch points of $L_{j'}\rightarrow B_q$, for $j,j'=1,\ldots,\ell$ with $j'\neq j$;

\item a copy of $\mathbb{P}^{1}$ at  $m_{0}$, denoted $L_0$ mapping to $B_q$ via a degree-$(\rho_{0}+1)$ map  ramified of order $\rho_0$ at $m_0$, simply ramified away from $m_0$ and unramified over the branch points of $L_{j'}\rightarrow B_q$, for $j'=1,\ldots,\ell$;

\item a copy  of $B$, mapping to $B$ isomorphically, at the point of $L_0$ distinct from $m_0$ lying over $q$ by the map $L_0\rightarrow B_q$.
\end{itemize}

At the end of this proccess we obtain from $C$ a nodal curve $C^{\prime}$ stably equivalent to
$C$,  from $B$ a nodal curve $B^{\prime}$ stably equivalent to
$B$, and a map
$\pi^{\prime}:C^{\prime}\rightarrow B^{\prime}$ of degree $k+1$ given by $\pi$ when restricted
to $C$, and by the maps described above when restricted to the added rational
components of $C^{\prime}$. Note that, since the map $L_0\rightarrow B_q$ has  $2(\rho_0+1)-2$ branch points, then there are at least two branch points in $B_q$ distinct from $q$. Therefore if $\pi$ is admissible, then $\pi'$ is also admissible.
\end{proof}

\subsection{Non-separating nodes}

We first focus on a partial normalization of a stable curve at a non-separating node. This will allow us to relate the gonality of an irreducible curve to that of its normalization in Corollary \ref{irred:cor}.

\begin{theorem}\label{irred:lem}
Let $C$ be a stable curve and
let $n$ be a non-separating node of $C$. Let $C_{n}$ be the normalization
of $C$ at $n$ and let
$\pi:C_{n}^{\prime}\rightarrow B$ be a finite map of degree $k$,
where $C_{n}^{\prime}$ is stably equivalent to $C_{n}$ and $\pi$ satisfies conditions (1), (3) and (4) of a $k$-sheeted admissible cover. 
Let  $\n1,\n2 \in C_n$ be the branches over $n$ and, for $i=1,2$, let ${n'}^{(i)}$ be a smooth point of $C_n'$ lying over $\ni$.
\begin{enumerate}[(a)]
\item If $\pi(\nlinha 1)  \neq\pi(\nlinha 2)$,
then there is a $(k+1)$-sheeted admissible cover $\pi':C'\rightarrow B'$, where $C'$ is stably equivalent to $C$, contains $C_n'$ as a subcurve and $\pi'|_{C_n'}=\pi$. In particular, $C$ is $(k+1)$-gonal.
\item If  $\pi({\nlinha 1})  =\pi({\nlinha 2})  $, then  there is a $k$-sheeted admissible cover $\pi':C'\rightarrow B'$, where $C'$ is stably equivalent to $C$, contains $C_n'$ as a subcurve and $\pi'|_{C_n'}=\pi$. In particular, $C$ is $k$-gonal.
\end{enumerate}
\end{theorem}
\begin{proof}
(a) 
To obtain the cover $\pi^{\prime}$ we proceed as follows.
Set $q_i:=\pi(\nlinha i)$ and let $l_{i}$  be the ramification index of $\pi$ at $\nlinha i$, for $i=1,2$. We thus have
\[\pi^{-1}(q_i)   = l_{i}\nlinha i+\lambda^{(i)}_{1}\mi_{1}+\ldots+\lambda^{(i)}_{u_i}\mi_{u_i},\]
where $\mi_j\in C_{n}'$, for  $i=1,2$ and $j=1,\ldots,u_i$. For each $i$ we glue (see Figure \ref{fig:irredlema}):
\begin{itemize}
\item a copy of $\Pbb^1$ to $B$ at $q_i$;

\item a copy of $\mathbb{P}^{1}$  to $C_n'$ at $\nlinha i$ mapping to the copy at $q_i$ via a degree-$(l_{i}+1)$ map ramified to order $l_i$ at $\nlinha i$  and simply ramified away from $\nlinha i$ (call this copy $L_i$);

\item a copy of $B$  passing through $L_1$ and $L_2$ at the points   lying over $q_1$ and $q_2$ but distinct from $\nlinha 1$ and $\nlinha 2$, respectively,  mapping to $B$ via an isomorphism;

\item a copy of $\mathbb{P}^{1}$ at every $\mi_{j}$, for $j=1,\ldots,u_i$, mapping to the copy at $q_i$ via a degree-$\lambda^{(i)}_{j}$ map totally ramified at $\mi_j$ and simply ramified away from  $\mi_j$.
\end{itemize}

At the end of this proccess we obtain from $C_n'$ a curve $C^{\prime}$ stably equivalent to
$C$,  from $B$ a  curve $B^{\prime}$ stably equivalent to $B$, and a map
$\pi^{\prime}:C^{\prime}\rightarrow B^{\prime}$ given by $\pi$ when restricted
to $C_{n}',$ and by the maps described above when restricted to the added rational
components of $C^{\prime}$. 
By construction, $C_n'$ is a subcurve of $C'$ and $\pi^{\prime}$ satisfies conditions (1) and (4) of a $(k+1)$-sheeted admissible cover. 
It also satisfies condition (3) since, for $i=1,2$, 
the restriction of $\pi'$ to $L_i$ has $2(l_i+1)-2$ ramification points and thus there are at least two marked points in the copy of $\mathbb P^1$ at $q_i$. 
The result now  follows from Theorem \ref{prop:step2} (a).

\begin{figure}[h!]
\centering
\includegraphics[height=2.6in]{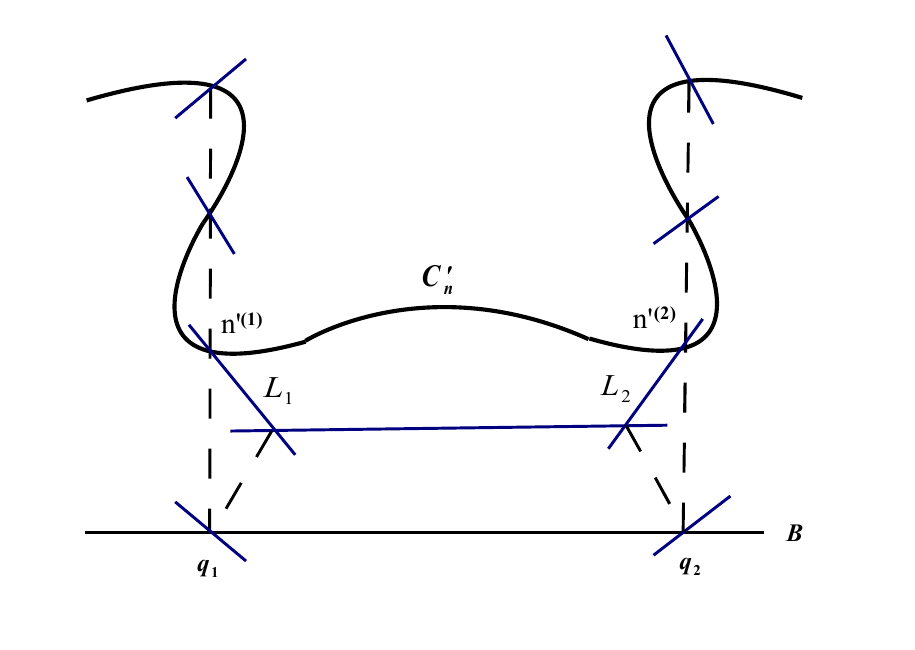}%
\caption{Theorem \ref{irred:lem} (a) for $k=3$}
\label{fig:irredlema}
\end{figure}

\smallskip
(b) To obtain the cover $\pi^{\prime}$ we proceed as follows.
Set $q:=\pi(\nlinha 1)=\pi(\nlinha 2) $ and let $l_{i}$ be
the ramification index of $\pi$ at $\nlinha i$, for $i=1,2$, so that
\[
\pi^{-1}\left(  q\right)  =l_{1}\nlinha 1+l_{2}\nlinha 2+\lambda_{1}m_{1}+\ldots+\lambda_{u}m_{u},
\]
where $m_{1},\ldots,m_{u}\in C_{n}'$. We glue (see Figure \ref{fig:irredlemb}):
\begin{itemize}
\item a copy of $\Pbb^1$ to $B$ at $q$;

\item a copy of $\mathbb{P}^{1}$ to $C_n'$ passing through  $\nlinha 1$ and  $\nlinha 2$ mapping to the copy over $q$ via a degree-$(l_{1}+l_{2})$ map ramified to order $l_1$ at $\nlinha 1$ and to order $l_2$ at $\nlinha 2$, and simply ramified away from $\nlinha 1$ and $\nlinha 2$ (call this copy $L$);

\item a copy of $\mathbb{P}^{1}$ at every $m_{j}$, for $j=1,\ldots,u$
mapping to the copy at $q$ via a degree-$\lambda_{j}$ map totally ramified at $m_j$ and simply ramified away from $m_j$.
\end{itemize}

\begin{figure}[h]
\centering 
\includegraphics[height=2.1in]{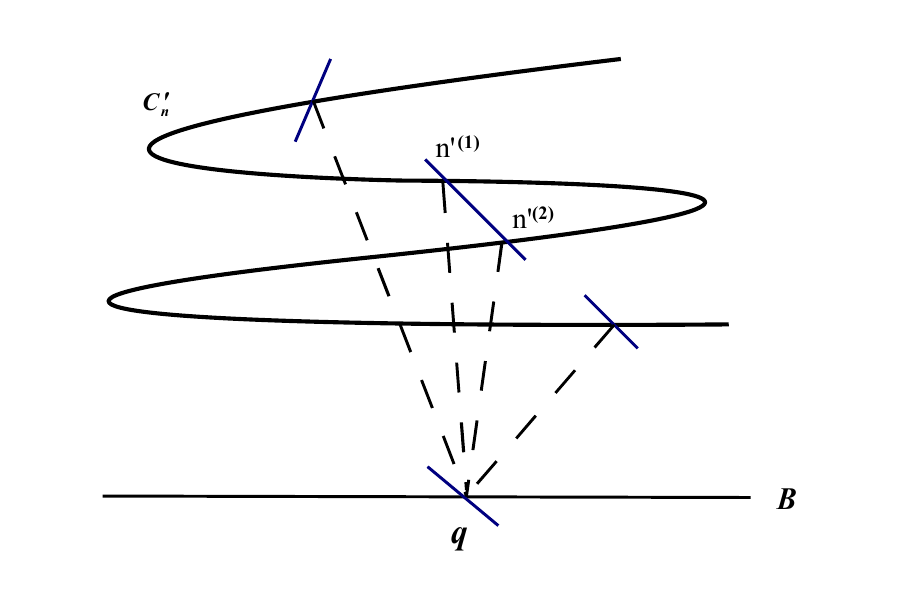}
\caption{Theorem \ref{irred:lem} (b) for $k=4$}
\label{fig:irredlemb}
\end{figure}

As before,  we obtain a degree-$k$ map $\pi^{\prime}:C^{\prime}\rightarrow B^{\prime}$ satisfying conditions (1) and (4) of an admissible cover, and such that $C_n'$ is a subcurve of $C'$ and $C'$ is stably equivalent to $C$. 
The map $\pi'$ also satisfies condition (3), since the restriction of $\pi'$ to $L$ has $2(l_1+l_2)-2$  ramification points, and hence there are at least two marked points in the copy of $\mathbb P^1$ at $q$. 
The result now  follows from Theorem \ref{prop:step2} (a).
\end{proof}

\begin{corollary}\label{irred:cor}
Let $C$ be an irreducible nodal curve with nodes
$n_1,\ldots,n_{\delta}$ and let $\widetilde{C}$ be its normalization. 
Let $\n1_{j},\n2_{j} \in \widetilde C$ be the branches over each $n_{j}$, for
$j=1,\ldots,\delta$.
Assume that $\widetilde{C}$ is $k$-gonal with degree-$k$ map
$\pi:\widetilde{C}\rightarrow \Pbb^1$ and let $\varepsilon$ be the number of indices $j$ such that $\pi(\n1_j)\neq\pi(\n2_j)$. Then  $C$ is $(k+\varepsilon)$-gonal. 
\end{corollary}
\begin{proof}
By Theorem \ref{HarrMum} it is sufficient to construct a $(k+\varepsilon)$-sheeted admissible cover $\pi^{\prime}:C^{\prime
}\rightarrow B^{\prime}$ where $C^{\prime}$ is stably equivalent to $C$. Let 
\[J:=\{j\in\{1,\ldots,\delta\} \ |\ \pi(\n1_j)\neq\pi(\n2_j)\}.\]
We  proceed first by induction on the number $\varepsilon$ of elements of  $J=\{j_1,\ldots,j_{\varepsilon}\}$.

Let $C_{0}=\widetilde C$ and for each $e=1,\dots,\varepsilon$ denote by $C_{e}$ the curve obtained by normalizing all the nodes of $C$, except $n_{j_1},\ldots,n_{j_e}$.
For  $0\leq e<\varepsilon$, assume there is a  
$(k+e)$-admissible cover
$\pi_{e}:C_{e}'\rightarrow B_{e}$ such  that $C_{e}'$ is a curve stably equivalent to $C_{e}$ having $\widetilde C$ as a component and 
$\pi_{e}|_{\widetilde C}=\pi$.  Then, by Theorem \ref{irred:lem} (a), there is a $(k+e+1)$-admissible cover
$\pi_{{e+1}}:C_{{e+1}}'\rightarrow B_{{e+1}}$ such that that $C_{{e+1}}'$ is stably equivalent to $C_{{e+1}}$, has $\widetilde C$ as a component and $\pi_{{e+1}}|_{\widetilde C}=\pi$. 

We have thus showed that there is a $(k+\varepsilon)$-sheeted admissible cover
$\pi_{{\varepsilon}}:C_{{\varepsilon}}'\rightarrow B_{{\varepsilon}}$ such that that $C_{{\varepsilon}}'$ is stably equivalent to the curve obtained normalizing the nodes $n_j$ of $C$ such that $\pi(\n1_j)=\pi(\n2_j)$. By construction, $\widetilde C$ is a component of $C_{{\varepsilon}}'$  and $\pi_{{\varepsilon}}|_{\widetilde C}=\pi$. Hence, for every index $j$ in
\[J':=\{j\in\{1,\ldots,\delta\} \ |\ \pi(\n1_j)=\pi(\n2_j)\}\]
we have $\pi_{{\varepsilon}}(\n1_j)=\pi_{{\varepsilon}}(\n2_j)$. Note that $J'$ has $\delta-\varepsilon$ elements.

It follows from  Theorem \ref{irred:lem} (b), applied $(\delta-\varepsilon)$ times, that there is a  $(k+\varepsilon)$-sheeted admissible cover
$\pi':C'\rightarrow B'$ such that that $C'$ is stably equivalent to $C$. 
\end{proof}

\begin{corollary}\label{cor:irred2}
Let $C$ be an irreducible nodal curve. If there exists a degree-$k$ map $\pi:C\rightarrow\mathbb{P}^{1} $ then $C$ is $k$-gonal.
\end{corollary}
\begin{proof}
Let $\nu:\widetilde{C}\rightarrow C$ be the normalization of $C$. For each node $n_j$ of $C$ let $\n1_{j},\n2_{j}\in \widetilde C$ be the branches over  $n_{j}$.
Then $\tilde \pi:=\pi\circ\nu$ is a degree-$k$ map such that $\tilde\pi(\n1_j)=\tilde\pi(\n2_j)$ for every $j$.
The result thus follows from Corollary \ref{irred:cor}, with $\varepsilon=0$.
\end{proof}

\subsection{Separating nodes}

We now turn our attention to separating nodes. 

\begin{theorem}\label{red:nonconjug0}
Let $C$ be a stable curve and let $Y_{1},\ldots,Y_r\subset C$ be connected  subcurves such that $C=Y_1\cup\ldots\cup Y_r$ and $Y_i\cap Y_{i'}$ is  finite, possibly empty, for $1\leq i\neq i'\leq r$.
For $i=1,\ldots,r$ let:
\begin{itemize}
\item  $\pi_{i}:Y_{i}'\rightarrow B_{i}$ be a $k_{i}$-sheeted admissible cover, where $Y_{i}^{\prime}$ is stably equivalent to $Y_{i}$;
\item  $n_{i,1},\ldots,n_{i,\delta_i}\in C$ be the intersection points between $Y_{i}$ and $Y_{i}^c$;
\item  $\ni_{j}$ be the branch over $n_{i,j}$ at $Y_{i}$ and let ${n_j'}^{(i)}$ be a smooth point of $ Y_i'$ lying over $\ni_j$, for $j=1,\ldots,\delta_i$.
\end{itemize}
If for every $i=1,\dots,r$ we have $\pi_{i}({n_1'}^{(i)})=\ldots=\pi_{i}({n_{\delta_i}'}^{\hspace{-.04in}(i)})$,
then $C$ is $(k_{1}+\ldots+k_{r}-\delta)  $-gonal, where $\delta=(\delta_1+\ldots+\delta_r)/2$.
\end{theorem}
\begin{proof}
By Theorem \ref{HarrMum} it is sufficient to construct a $(k_{1}+\ldots+k_{r}-\delta)$-sheeted admissible cover $\pi:C^{\prime}\rightarrow B$
where $C^{\prime}$ is stably equivalent to $C$. 
To obtain the cover we proceed as follows.

For $i=1,\ldots,r$ let $q_{i}:=\pi_{i}({n_1'}^{\hspace{-.02in}(i)})=\ldots=\pi_{i}({n_{\delta_i}'}^{\hspace{-.04in}(i)})$. 
For $j=1,\ldots,\delta_i$ let $l_j^i$ be
the ramification index of $\pi_i$ at  ${n_j'}^{(i)}$, so that
\[
(\pi_{i})^{-1}(q_{i})=
l_{1}^{i}{n_1'}^{(i)} +\ldots+l_{\delta_i}^{i}{n_{\delta_i}'}^{\hspace{-.04in}(i)}+
\lambda_{1}^{i}\mi_1+\ldots+\lambda_{u_i}^{i}\mi_{u_i},
\]
where $\mi_1,\ldots,\mi_{u_i}\in Y_{i}'$. We glue (see Figure \ref{fig:prop6r3}):
\begin{itemize}
\item  a copy of $\mathbb{P}^{1}$, denoted by $B^{^{\prime}}$, passing through
$B_{1},\ldots,B_r$ at $q_{1},\ldots,q_r$, respectively, and thus linking the curves.  Denote by $B$ the genus-$0$ curve thus obtained;

\item whenever $n_{i_0,j_0}=n_{i_1,j_1}$, a copy of $\mathbb{P}^{1}$ passing through $Y_{i_0}'$ and $Y_{i_1}'$ at 
${n_{j_0}'}^{\hspace{-.04in}(i_0)}$ and ${n_{j_1}'}^{\hspace{-.04in}(i_1)}$,  and thus linking both curves,
mapping to $B^{\prime}$ via a degree-$(l_{j_0}^{i_0}+l_{j_1}^{i_1}-1)$ map ramified to order $l_{j_0}^{i_0}$ at ${n_{j_0}'}^{\hspace{-.04in}(i_0)}$ 
and to order $l_{j_1}^{i_1}$ at ${n_{j_1}'}^{\hspace{-.04in}(i_1)}$ and simply ramified elsewhere, where  $1\leq i_0,i_1\leq r$, $1\leq j_s\leq \delta_{i_s}$ and $s=1,2$ (call this copy $L_{j_0,j_1}^{i_0,i_1}$);

\item 
 a copy of $B_t$, mapping to  $B_t$ isomorphically,
at each point $m$ of $L_{j_0,j_1}^{i_0,i_1}$, distinct from ${n_{j_0}'}^{\hspace{-.04in}(i_0)}$ and ${n_{j_1}'}^{\hspace{-.04in}(i_1)}$,  
lying over $q_t$  by the map $L_{j_0,j_1}^{i_0,i_1}\rightarrow B'$, 
where  $1\leq i_0,i_1,t\leq r$ and  $1\leq j_s\leq \delta_{i_s}$, for $s=1,2$;

\item a copy of $\mathbb{P}^1$ at $\mi_{j}$, mapping to $B'$ via a degree-$\lambda_{j}^{i}$ map totally ramified at $\mi_j$,
unramified over $q_t$ for every $t\neq i$, 
and simply ramified elsewhere, 
for each $1\leq i,t\leq r$ and $1\leq j\leq u_i$ (call this copy $L_j^i$);

\item a copy of $B_t$ 
at each point $m$ of $L_j^i$ lying over $q_t$ by the map $L_{j}^i\rightarrow B'$, mapping to  $B_t$ isomorphically,
for each $1\leq i,t\leq r$ with $t\neq i$, and $1\leq j\leq u_i$.
\end{itemize}

\begin{figure}[h]
\centering
\includegraphics[height=2.2in]{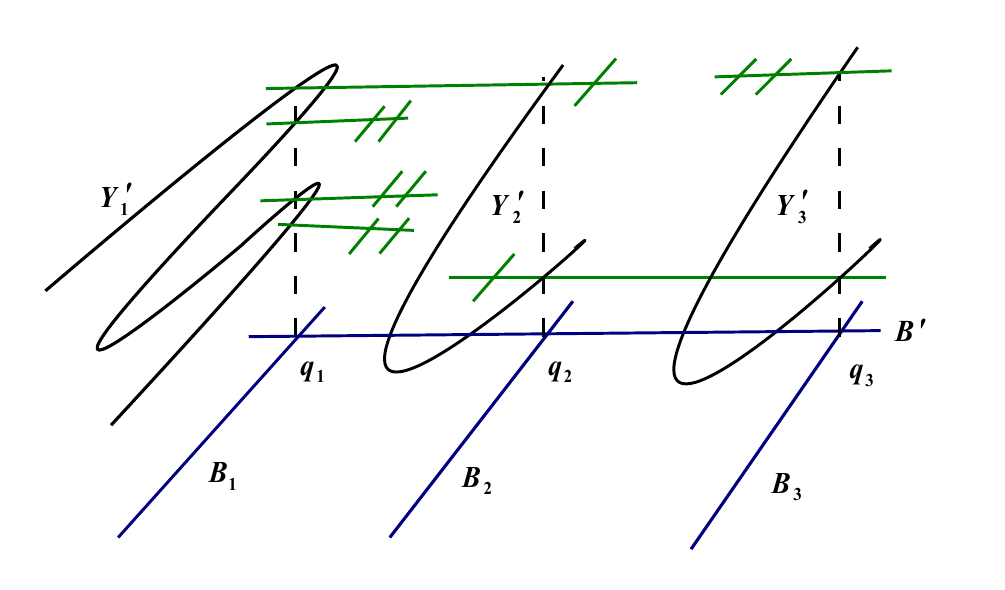}
\caption{Theorem \ref{red:nonconjug0} for $r=3$,  $\delta=2$, $k_1=4$ and $k_2=k_3=2$}
\label{fig:prop6r3}
\end{figure}

We thus obtain a nodal curve $C^{\prime}$ stably equivalent to $C$
 and a map $\pi:C^{\prime}\rightarrow B$ given by $\pi_{i} $ when
restricted to $Y_{i}'$, and by the maps described above when restricted to the
added rational components of $C^{\prime}$. 
By construction, $\pi$ has degree $k=k_{1}+\ldots+k_{r}-\delta$ and satisfies conditions (1) and (4) of an admissible cover.
As for condition (3), we  note that if $r\geq 3$ then $B$ is clearly stable. 

Now we examine the case $r=2$. 
We first note that, since  $\pi_{i_0}$ and $\pi_{i_1}$ are admissible covers, the map $L_{j_0,j_1}^{i_0,i_1}\rightarrow B'$   has degree at most 3.
Moreover, this map has degree equal to 3 
if and only if 
$l_{j_0}^{i_0}=2$ and $l_{j_1}^{i_1}=2$, that is, 
if and only if the map is ramified both at ${n_{j_0}'}^{\hspace{-.04in}(i_0)}$ and at ${n_{j_1}'}^{\hspace{-.04in}(i_1)}$. Since the ramifications are simple and, by Riemann-Hurwitz, the map has four ramification points, two of which are the nodes $q_{i_0}$ and $q_{i_1}$, we conclude that $B$ is stable in this case.

Now, $L_{j_0,j_1}^{i_0,i_1}\rightarrow B'$   has degree   2 if and only if either
$l_{j_0}^{i_0}=2$ and $l_{j_1}^{i_1}=1$, or 
$l_{j_0}^{i_0}=1$ and $l_{j_1}^{i_1}=2$, that is, 
if and only if the map is ramified at ${n_{j_0}'}^{\hspace{-.04in}(i_0)}$ and unramified at ${n_{j_1}'}^{\hspace{-.04in}(i_1)}$, or vice-versa. Again, since
the ramifications are simple and, by Riemann-Hurwitz, the map $L_{j_0,j_1}^{i_0,i_1}\rightarrow B'$  has two ramification points,  exactly one of which is a node of $B$, we conclude that $B$ is stable.

\begin{figure}[h]
\centering
\includegraphics[height=2.2in]{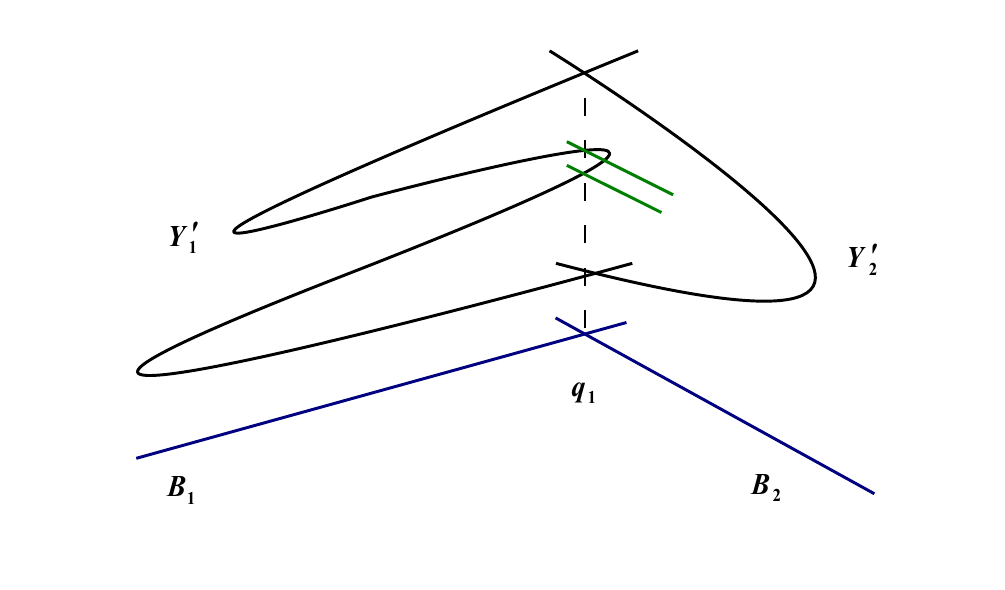}
\caption{Theorem \ref{red:nonconjug0} for $r=2$,  $\delta=2$, $k_1=4$ and $k_2=2$}
\label{fig:prop6r2}
\end{figure}

Finally, if the maps $L_{j_0,j_1}^{i_0,i_1}\rightarrow B'$  are isomorphisms, then there are no ramification points. In this case,
contracting $B'$ and the rational components of $C'$ mapping to it, one also obtains condition (3), see Figure \ref{fig:prop6r2}. The result now follows from Theorem \ref{prop:step2} (a).
\end{proof}

\begin{corollary}\label{cor:sepnode}
Let $C$ be a stable curve and let $n$ be a separating node of $C$ with associated tails $Y_1$ and $Y_2$. 
If $Y_i$ is $k_i$-gonal for $i=1,2$ then $C$ is $(k_1+k_2-1)$-gonal.
\end{corollary}
\begin{proof}
Follows from  Theorem \ref{red:nonconjug0}.
\end{proof}

\begin{theorem}\label{teo:uplowbound}
Let $C$ be a stable curve and let $Y_1$ be a connected subcurve of $C$.  Assume $Y_2:=Y_1^c$ is connected and let $\delta$ be the number of nodes in $Y_1\cap Y_2$.
If $Y_i$ is $k_i$-gonal for $i=1,2$ then $C$ is $k$-gonal for some $k$ satisfying 
$k_1+k_2-\delta\leq k\leq k_1+k_2+\delta-2$.
\end{theorem}
\begin{proof}
Let $n_1,\ldots,n_\delta$ be the nodes in $Y_1\cap Y_2$. 
As in Theorem \ref{red:nonconjug0}, let
  $\pi_{i}:Y_{i}'\rightarrow B_{i}$ be a $k_{i}$-sheeted admissible cover, where $Y_{i}^{\prime}$ is stably equivalent to $Y_{i}$.  Let
  $\ni_{j}$ be the branch over $n_{j}$ on $Y_{i}$ and let ${n_j'}^{(i)}$ be a smooth point on $Y_i'$ lying over $\ni_j$, for $j=1,\ldots,\delta$ and $i=1,2$.

For $i=1,2$ let $q_i=\pi_i({n_1'}^{(i)})$ and 
\[J=\{j\ |\ \pi_1({n_j'}^{(1)})=q_1\text{ and }  \pi_2({n_j'}^{(2)})=q_2\}.\]
Let $C_J$ be the curve obtained by normalizing all the nodes $n_j$ of $C$ such that $j\not\in J$. 
If $\varepsilon$ is the number of indices in $J$, then $1\leq \varepsilon\leq \delta$ and, 
by Theorem \ref{red:nonconjug0}, $C_J$ is $(k_1+k_2-\varepsilon)$-gonal. 

Let $\gamma=\delta-\varepsilon$ and $\{j_1,\ldots,j_{\gamma}\}=\{1,\ldots,\delta\}\smallsetminus J$. 
Let $C_0=C$ and, for each $1\leq e\leq \gamma$, denote by  
$C_{e}$  the curve obtained normalizing the nodes $n_{j_1},\ldots,n_{j_e}$ of $C$.  Note that $C_\gamma=C_J$. The result now follows
 follows applying Theorem \ref{irred:lem} (a) consecutively to the nodes $n_{j_{\gamma}},\ldots,n_{j_1}$ of $C_{\gamma-1},\ldots,C_0$, respectively.
 \end{proof}

A nodal curve $C$ is called \emph{treelike} if the normalization of $C$ at its internal nodes is a curve of compact type.  Equivalently, a nodal curve  $C$ having $p$ irreducible components is treelike if the number of external nodes of $C$ is $p-1$.

\begin{corollary}\label{cor:bound}
Let $C$ be a stable curve with irreducible components $C_1,\ldots,C_p$ and $\delta$ external nodes.
If $C_i$ is $k_i$-gonal for $i=1,\ldots,p$ then $C$ is $k$-gonal for some $k$ satisfying $k_1+\ldots+k_p-\delta\leq k\leq k_1+\ldots+k_p+\delta-2(p-1)$.
In particular, if $C$ is a treelike curve, then $k=k_1+\ldots+k_p-\delta$.
\end{corollary}
\begin{proof}
Follows from Theorem \ref{teo:uplowbound}.
\end{proof}

The upper bound in Corollary \ref{cor:bound} can be improved under some conditions.

\begin{proposition}\label{red:conjugate}
Let $C$ be a stable curve with $p$ smooth irreducible components $C_1,\ldots,C_p$.
For each node $n_j$ of $C$, let $C_{j(1)}$ and $C_{j(2)}$   be the irreducible components of $C$ containing $n_j$ 
and denote by $\n1_j$ and $\n2_j$ the branches of $n_j$ on  $C_{j(1)}$ and $C_{j(2)}$, respectively.  
Assume that:
\begin{enumerate}
\item there is a degree-$k_s$ map $\pi_{s}:C_{s}\rightarrow \Pbb^1$, for each irreducible component $C_s$ of $C$;
\item  $\pi_{j(1)}(\n1_{j})=\pi_{j(2)}(\n2_{j})$ for every  node $n_j$ of $C$, 
\item $\pi_{j(i)}(\ni_{j})\neq\pi_{j'(i')}(\n{i'}_{j'})$ for every  pair of distinct nodes $n_j,\,n_{j'}$ of $C$ and $i,i'\in\{1,2\}$.
\end{enumerate}
Then $C$ is $(k_1+\ldots+k_p)-$gonal.
\end{proposition}
\begin{proof}
By Theorem \ref{HarrMum} it is
enough to construct a $\left(  k_{1}+\ldots+k_{p}\right)  $-sheeted
admissible cover $\pi:C^{\prime}\rightarrow B'$ where $C^{\prime}$ is stably
equivalent to $C$. 
Set $B=\Pbb^1$ so that the maps $\pi_{s}$ have image in $B$.
For each node $n_j$ of $C$ we proceed as follows. Set $q_{j}:=\pi_{j(i)}(\ni_j)\in B$ and let $l_{j(i)}$ be the ramification index of $\pi_{j(i)}$ at $\ni_j$, for $i=1,2$, so that
\[(\pi_{j(i)})^{-1}(q_{j}) = l_{j(i)}\ni_{j}+ \lambda_{j(i),1}m_{j(i),1}+\ldots+\lambda_{j(i),u_{i,j}}m_{j(i),u_{i,j}}.\]
where $m_{j(i),1},\ldots,m_{j(i),u_{i,j}}\in C_{j(i)}$. 
We glue (see Figure \ref{fig:prop11}):
\begin{itemize}
\item a copy of $\Pbb^1$ to $B$ at $q_{j}$;

\item a copy of $\mathbb{P}^{1}$ passing through $C_{j(1)}$ and $C_{j(2)}$ at
$\n1_{j}$ and $\n2_{j}$, thus 
linking the two curves, and 
 mapping to the copy at $q_{j}$
via a degree-$(  l_{j(1)}+l_{j(2)})$ map ramified to order $l_{j(1)}$ at $\n1_j$ and to order $l_{j(2)}$ at $\n2_j$
and simply ramified away from  $\n1_j$ and  $\n2_j$;

\item a copy of $\mathbb{P}^{1}$ at every $m_{j(i),s}$ 
mapping to the copy at $q_{j}$ via a degree-$\lambda_{j(i),s}$ map totally ramified at $m_{j(i),s}$ and simply ramified elsewhere.
\end{itemize}

Moreover, for each $j$ and each $i$ such that $i\not\in\{j(1),j(2)\}$ set
\[\pi_i^{-1}(q_j)=
\nu^{i,j}_1m^{i,j}_1+\ldots+
\nu^{i,j}_{u_{i,j}} m^{i,j}_{u_{i,j}},
\]
where $m^{i,j}_{1},\ldots,m^{i,j}_{u_{i,j}}\in C_i$.
Then we glue a copy of $\mathbb P^1$ at every $m^{i,j}_\ell$ mapping to the copy at $q_j$ via a degree-$\nu^{i,j}_\ell$ map totally ramified at $m^{i,j}_\ell$ and simply ramified elsewhere.

\begin{figure}[h!]
\centering
\includegraphics[height=2.00in]{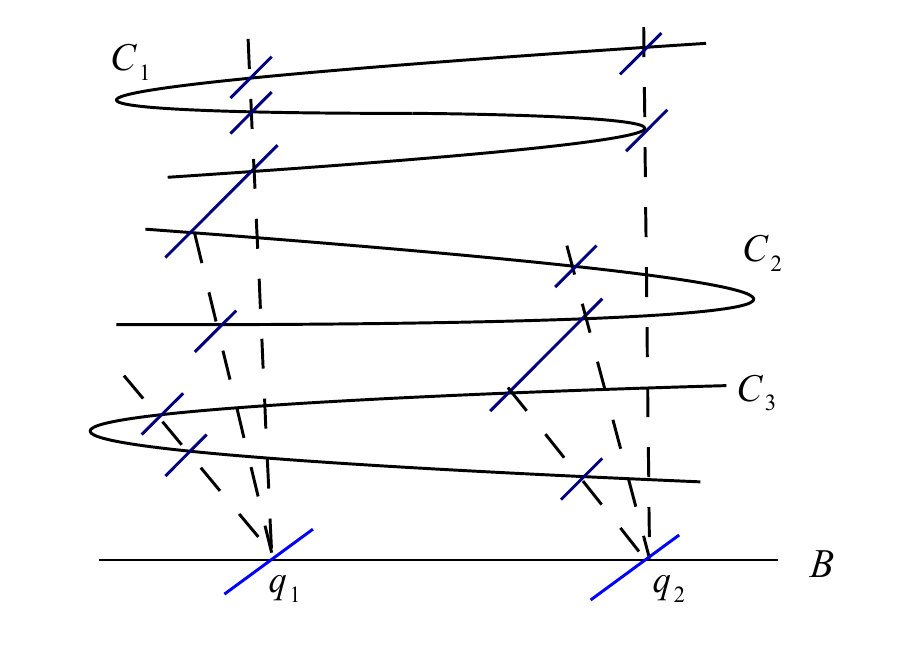} 
\caption{Proposition \ref{red:conjugate} for $p=3$, $k_1=3$ and $k_2=k_3=2$}
\label{fig:prop11}
\end{figure}

We thus obtain from $C$ a
curve $C^{\prime}$ stably equivalent to $C$, containing $C_s$ for every $s=1,\ldots,p$, 
and a map $\pi:C^{\prime}\rightarrow B'$ given by $\pi_{s} $ when
restricted to $C_{s}$ and by the maps described above when restricted to the
added rational components of $C^{\prime}$. By construction, $\pi$ is a map of degree $k_{1}+\ldots+k_{p}$  satisfying conditions (1), (3) and (4) of an admissible cover. The result then follows from Theorem \ref{prop:step2} (a).
\end{proof}

\section{Maps of Hurwitz schemes}\label{sec:hurwitzscheme}

The Hurwitz scheme $H_{k,g}$ in its simplest form parametrizes families of $k$-sheeted
covers of $\mathbb{P}^{1}$ with $b=2g+2k-2$ ordinary branch points. More precisely, 
$H_{k,g}$ is the moduli space of pairs $(  C,\pi)$, where 
$C$ is a smooth curve of genus $g$ and $\pi\colon C\rightarrow\mathbb{P}^{1}$ is finite of degree $k$
branched only over $b$ distinct points $Q_{1},\ldots,Q_{b}\in\mathbb{P}^{1}$.
By \cite{F} and \cite{M}, $H_{k,g}$ is itself a finite \'{e}tale cover of the open subset $U$ of $(\mathbb P^1)^{b-3}$ parametrizing 
 sequences $\{Q_{i}\}$ where $Q_i\neq Q_j$ for $i\neq j$.

The scheme $H_{k,g}$ maps naturally to the moduli space ${M}_{g}$ of
smooth curves of genus $g$ via a forgetful map $\sigma$ whose
 image contains an open subset of the locus in ${M}_{g}$ of $k$-gonal curves. A good compactification $\overline{H}_{k,g}$ of $H_{k,g}$
should extend $\sigma$ to a morphism $\overline{\sigma}\colon\overline{H}_{k,g}\rightarrow\overline{{M}}_{g}$, 
where $\overline{ M}_g$ is the moduli space of stable curves of genus $g$.

This compactification can be done by means of Knudsen's theory
of $b$-pointed curves, see \cite{K2} and \cite{HM82}. Knudsen introduced a smooth projective compatification $\mathcal P$ of $U$ parametrizing $b$-pointed stable curves $C$ of genus 0, 
up to isomorphism. We note that $U$ is an open dense subset of $\mathcal P$ corresponding to $b$-pointed irreducible curves of genus 0.

Let $\overline{\mathcal{H}}_{k,g}$ be the
functor that takes a scheme $S$ to the set of \emph{genus-$g$ $k$-sheeted admissible covers over $S$}, that is, the set of morphisms $\pi\colon \mathcal C\rightarrow \mathcal B$ where
$\mathcal C$ and $\mathcal B$ are families of  curves over $S$ of genera $g$ and $0$, respectively, and such that, for each $s\in S$, the map $\pi(s)\colon \mathcal C(s)\rightarrow \mathcal B(s)$ is a  $k$-sheeted admissible cover.
By  \cite[Theorem 4]{HM82} and \cite[Theorem 2.9]{M}, the functor $\overline{\mathcal{H}}_{k,g}$ is coarsely
represented by a projective irreducible scheme $\overline{H}_{k,g}$ finite over $\mathcal{P}$. 
In
particular,  $\overline{H}_{k,g}$ contains  ${H}_{k,g}$ as an open and dense subset and
\[\dim\overline{H}_{k,g}=2g+2k-5.\]

Now we consider Hurwitz schemes with  marked points. Let 
$\overline{\mathcal{H}}_{k,g,\ell_1,\ldots,\ell_n}$ be the functor taking a scheme $S$ 
to the set of \emph{genus-$g$ $k$-sheeted $(\ell_1,\ldots,\ell_n)$-pointed admissible covers over $S$}, that is,  the set of tuples 
\[(\pi\colon \mathcal C\rightarrow \mathcal B; p^1_{1},\ldots, p^1_{\ell_1},  \ldots,  p_1^n,\ldots, p_{\ell_n}^n)\] 
where
$\pi$ is a genus-$g$ $k$-sheeted admissible cover over $S$ 
and  \[(\mathcal C,p^1_{1},\ldots,p^1_{\ell_1},\ldots,
p_1^n,\ldots,p_{\ell_n}^n)\] is a $(\ell_1+\ldots+\ell_n)$-pointed stable curve of genus $g$ over $S$  such that 
\begin{equation}\label{eq:pointedhurwtz}
\pi(p^i_{1})=\ldots=\pi(p^i_{\ell_i})
\quad \text{and}\quad
\pi(p^i_{1})\neq\pi(p^j_{1})
\end{equation}
for $1\leq i\neq j\leq n$.
Let also ${\mathcal H}_{k,g,\ell_1,\ldots,\ell_n}$ be  the subfunctor  of $\overline{\mathcal{H}}_{k,g,\ell_1,\ldots,\ell_n}$ 
taking a scheme $S$ 
to the set of genus-$g$ $k$-sheeted $(\ell_1,\ldots,\ell_n)$-pointed admissible covers over $S$ as above such that $\mathcal C$ and $\mathcal B$ are smooth curves over $S$. 
Moreover, let  $\overline{\mathcal H}_{k,g,\ell_1,\ldots,\ell_n}^{un}$ (resp. ${\mathcal H}_{k,g,\ell_1,\ldots,\ell_n}^{un}$) be  the subfunctor  of $\overline{\mathcal{H}}_{k,g,\ell_1,\ldots,\ell_n}$ (resp. ${\mathcal{H}}_{k,g,\ell_1,\ldots,\ell_n}$)
taking a scheme $S$ 
to the set of genus-$g$ $k$-sheeted $(\ell_1,\ldots,\ell_n)$-pointed admissible covers over $S$ as above such that the admissible cover is unramified over the marked points.

\begin{theorem}\label{teo:repr}
For $g,k\geq 2$ and $1\leq\ell_1,\ldots,\ell_n\leq k$, 
the functor $\overline{\mathcal{H}}_{k,g,\ell_1\ldots,\ell_n}$ is coarsely represented by a scheme $\overline{H}_{k,g,\ell_1,\ldots,\ell_n}$ of dimension $2g+2k-5+n$.  

Moreover, ${\mathcal H}_{k,g,\ell_1,\ldots,\ell_n}$ (resp. $\overline{\mathcal H}_{k,g,\ell_1,\ldots,\ell_n}^{un}$, resp. ${\mathcal H}_{k,g,\ell_1,\ldots,\ell_n}^{un}$) 
is coarsely represented by a scheme ${H}_{k,g,\ell_1,\ldots,\ell_n}$ (resp. $\overline{H}_{k,g,\ell_1,\ldots,\ell_n}^{un}$, resp. ${H}_{k,g,\ell_1,\ldots,\ell_n}^{un}$) open and dense in $\overline{H}_{k,g,\ell_1,\ldots,\ell_n}$. 
\end{theorem}
\begin{proof}
The proof of the existence of $\overline{H}_{k,g,\ell_1,\ldots,\ell_n}$
follows the steps  of \cite[Theorem 4, p. 58]{HM82}. 
Now, consider the  map 
\[ 
\zeta:\overline H_{k,g,\ell_1,\ldots,\ell_n} \rightarrow \overline H_{k,g}\times_{\overline M_g}\overline M_{g,n}
\]
taking an admissible cover $\pi\colon C\rightarrow B$ with marked points
$ p^1_{1},\ldots, p^1_{\ell_1}, 
\ldots, p_1^n,\ldots, p_{\ell_n}^n$ satisfying \eqref{eq:pointedhurwtz}
on $C$ to the same admissible cover $\pi\colon C\rightarrow B$  together with the stabilization of  $(C, p^1_{1},\ldots,p_1^n)$, that is, the pointed stable curve % $(C', p^1_{1},\ldots,p_1^n)$ 
obtained contracting the rational components $E$ of $C$ such that the number of points where $E$ meets $E^c$ plus the number of indices $i$ such that $p_1^i$ lies on $E$ is smaller than three. 
Since $\zeta$ is dominant and quasi-finite, then the dimension of
$\overline H_{k,g,\ell_1,\ldots,\ell_n}$ is equal to that of $\overline H_{k,g}\times_{\overline M_g}\overline M_{g,n}$, that is,
 $2g+2k-5+n$.

For the last assertion, it suffices to note that the set of smooth curves is open and dense in the moduli of stable curves. Moreover, given an admissible cover $\pi\colon C\rightarrow B$, the set of points in $C$ over which $\pi$ is unramified is open and dense in $C$. 
\end{proof}

We now turn our attention to the morphisms of Hurwitz schemes induced by the admissible covers produced in Section 3. 
First let 
\[
\Gamma\colon\overline{\mathcal{H}}_{k,g,1}^{un}\rightarrow\overline{\mathcal{H}}_{k+1,g}
\]
be the  natural transformation  given, for a scheme $S$, by taking a $k$-sheeted admissible cover $\pi\colon\mathcal C\rightarrow \mathcal B$ 
over $S$ where $\mathcal C$ is a genus-$g$ curve over $S$ with one marked point to  a $(k+1)$-sheeted admissible cover 
$\pi'\colon\mathcal C'\rightarrow \mathcal B^{\prime}$ where $\mathcal C'$,  $\mathcal B'$ and  $\pi'$ are constructed as in Theorem \ref{prop:step2} (b).

Let 
\[
\Phi\colon\overline{\mathcal{H}}^{un}_{k,g,1,1}\rightarrow\overline{\mathcal{H}}_{k+1,g+1}
\]
be the  natural transformation  given, for a scheme $S$, by taking a $k$-sheeted admissible cover $\pi\colon\mathcal C\rightarrow \mathcal B$ 
over $S$ where $\mathcal C$ is a genus-$g$ curve over $S$ with two marked points $p_1$ and $p_2$ such that $\pi(p_1)\neq\pi(p_2)$ 
to 
 a $(k+1)$-sheeted admissible cover 
$\pi'\colon\mathcal C'\rightarrow \mathcal B^{\prime}$ where $\mathcal C'$,  $\mathcal B'$ and  $\pi'$ are constructed as in Theorem \ref{irred:lem} (a).

Let 
\[
\Psi\colon\overline{\mathcal{H}}^{un}_{k,g,2}\rightarrow\overline{\mathcal{H}}_{k,g+1}
\]
be the natural transformation given, for a scheme $S$, by taking a $k$-sheeted admissible cover $\pi\colon\mathcal C\rightarrow \mathcal B$ 
 over $S$ where $\mathcal C$ is a genus-$g$ curve over $S$ with two marked points $p_1$ and $p_2$ such that $\pi(p_1)=\pi(p_2)$ 
to 
 a $k$-sheeted admissible cover 
$\pi'\colon\mathcal C'\rightarrow \mathcal B^{\prime}$ where $\mathcal C'$,  $\mathcal B'$ and  $\pi'$ are constructed as in Theorem \ref{irred:lem} (b).

And finally, let
\[\Lambda\colon \overline{\mathcal{H}}_{k_{1},g_{1},\delta}^{un}
\times \overline{\mathcal{H}}_{k_{2},g_{2},\delta}^{un} 
\rightarrow \overline{\mathcal{H}}_{k_{1}+k_{2}-\delta,g_{1}+g_{2}+\delta-1}\]
be the   natural transformation given, for a scheme $S$, by taking a $k_i$-sheeted admissible cover $\pi_i\colon \mathcal C_i\rightarrow \mathcal B_i$ over $S$ where $\mathcal C_i$ is a genus-$g_i$ curve over $S$ with $\delta$ marked points $n_1^{(i)}, \ldots,n_\delta^{(i)}$ such that $\pi_i(n_1^{(i)})= \ldots=\pi_i(n_\delta^{(i)})$, for $i=1,2$, to a $(k_1+k_2-\delta)$-sheeted admissible cover $\pi\colon\mathcal C\rightarrow\mathcal B$ where $\mathcal C$,  $\mathcal B$ and $\pi$ are constructed as in Theorem \ref{red:nonconjug0} for $r=2$.

We note  that in order to view the  construction of Theorem \ref{red:nonconjug0} for gluing
 $r\geq 3$  admissible covers $\pi_i:C_i\rightarrow B_i$  as a natural transformation, one would need to consider some extra data to take into account  the choice of the gluing, that is, which marked points of each   $C_i$ should be glued to which marked points of each $C_{j}$, for $i,j=1,\ldots, r$. 
Each choice would thus give a different natural transformation $\Lambda$, rendering the notation  very cumbersome.
Although we believe that, as in the case $r=2$ treated in the next result, every such $\Lambda$ should induce a generically injective morphism of pointed Hurwitz schemes, the general case for any $r$ will not be treated here.

\begin{theorem}
\label{thm:geninj}
For $g\geq 2$ and $k,\delta\geq 1$,
the natural transformations $\Gamma$, $\Phi$, $\Psi$ and $\Lambda$ induce generically injective morphisms
\[\gamma\colon \overline{H}_{k,g,1}^{un}\rightarrow\overline{H}_{k+1,g},
\quad\quad
\phi\colon \overline{H}_{k,g,1,1}^{un}\rightarrow\overline{H}_{k+1,g+1},
\quad\quad
\psi\colon \overline{H}_{k,g,2}^{un}\rightarrow\overline{H}_{k,g+1},\]
and
\[\lambda\colon Symm^2(\overline{H}_{k,g,\delta}^{un})
\rightarrow \overline{H}_{2k-\delta,2g+\delta-1},\]
where $Symm^2(\overline{H}_{k,g,\delta}^{un})$ is the symmetric product. For $g_1,g_2\geq2$ and $k_1,k_2\geq 1$ such that $(g_1,k_1)\neq(g_2,k_2)$,
the natural transformation  $\Lambda$ induces a generically injective morphism
\[\lambda\colon \overline{H}_{k_{1},g_{1},\delta}^{un}
\times
\overline{H}_{k_{2},g_{2},\delta}^{un} 
\rightarrow 
\overline{H}_{k_{1}+k_{2}-\delta,g_{1}+g_{2}+\delta-1},\]
\end{theorem}
\begin{proof}
For simplicity, we denote by $\nu\colon\overline{H}_{\underline n}^{un} \rightarrow  \overline{H}_{\underline n'}$ the induced morphism 
$\gamma$, $\phi$ or $\psi$. 
Since $H_{\underline n}^{un}$ is an open dense subset of $\overline H_{\underline n}^{un}$, 
to show that $\nu$ is generically injective, 
it is enough to show that  $\nu^{-1}(\nu(h))=\{h\}$ for $h\in H_{\underline n}^{un}$. 
Let  $\pi\colon C\rightarrow B$ be the admissible cover associated to $h$, where $B=\mathbb{P}^1$. 
Let $\pi'\colon C'\rightarrow B'$ be the admissible cover associated to $h'=\nu(h)$. By construction, $C$ is a subcurve  of $C'$ and we have  $\pi'|_C=\pi$ and $g(C')=g(C)+\varepsilon(\nu)$, where $\varepsilon(\nu)=1$ for $\nu\in\{\phi,\psi\}$ and  $\varepsilon(\gamma)=0$. 

Let $\tilde h\in \nu^{-1}(\nu(h))$ and let $\tilde \pi\colon \widetilde C\rightarrow \widetilde B$ be the admissible cover associated to $\tilde h$.
Then $\widetilde C$ is a subcurve of $C'$, $g(C')=g(\widetilde C)+\varepsilon(\nu)$ and $\pi'|_{\widetilde C}=\tilde \pi$. 
Since $C$ is smooth and both $C$ and $\widetilde C$ are  subcurves of $C'$ of same genus $g(C')-\varepsilon(\nu)\geq 2$, we have that $C$ is a smooth component of $\widetilde C$ and  $\tilde \pi|_C= \pi$.

Now note that since  $h'=\nu( h)$, then by construction the number of irreducible components of $B'$ is equal to the number of irreducible components of $B$ plus one (if $\nu\in\{\gamma,\psi\}$) or  two (if $\nu=\phi$). 
But also, since   $h'=\nu(\tilde h)$, the number of irreducible components of $B'$ is equal to the number of irreducible components of $\widetilde B$ plus one (if $\nu\in\{\gamma,\psi\}$) or  two (if $\nu=\phi$). Hence $B$ and $\widetilde B$  have the same number of components and since $B=\mathbb{P}^1$, we have  $\widetilde B=\mathbb{P}^1$. But then, by condition (1) of an admisible cover, $\widetilde C$ must be smooth and  we have $\widetilde C=C$ and $\tilde \pi=\pi$, showing that $\tilde h=h$.

We now show that 
$\lambda\colon \overline{H}^{un}_{\underline n_1}\times \overline{H}^{un}_{\underline n_2} \rightarrow  \overline{H}_{\underline n'}$ 
is generically injective.  
Again, it is enough to show that if $h_i\in H^{un}_{\underline n_i}$ for $i=1,2$, then $\lambda^{-1}(\lambda(h_1,h_2))=\{h_1,h_2\}$. 
Let  $\pi_i\colon C_i\rightarrow B_i$ be the admissible cover associated to $h_i$, where $B_i=\mathbb{P}^1$, for $i=1,2$.
Let $\pi'\colon C'\rightarrow B'$ be the admissible cover associated to $h'=\lambda(h_1,h_2)$. 
By construction, $B'$ has  two components (corresponding to $B_1$ and $B_2$),  $C_i$  is  a subcurve of $C'$ and we have   $\pi'|_{C_i}=\pi_i$, for $i=1,2$. 

For $i=1,2$ let $\tilde h_i\in \overline H_{\underline n_i}^{un}$ be such that $\lambda(\tilde h_1,\tilde h_2)=h'$ and let $\tilde \pi_i\colon \widetilde C_i\rightarrow \widetilde B_i$ be the admissible cover associated to $\tilde h_i$.
Then  $\widetilde C_i$ is a subcurve of $C'$ and
 $\pi|_{\widetilde C_i}=\tilde \pi_i$, for $i=1,2$.  Moreover,  the number of components of $B'$ is the sum of the numbers of components of $\widetilde B_1$ and $\widetilde B_2$. But $B'$ has only two components, and hence $\widetilde B_i=\mathbb{P}^1$ for $i=1,2$. 
Then, by condition (1) of an admissible cover, both $\widetilde C_1$ and $\widetilde C_2$ are smooth curves. By the construction on Theorem \ref{red:nonconjug0}, we must have $\{C_1,C_2\}=\{\widetilde C_1,\widetilde C_2\}$. 
Now, if $\underline n_1\neq \underline n_2$ then this clearly implies  $h_i=\tilde h_i$ for $i=1,2$. If $\underline n_1= \underline n_2=\underline n$ then $\{h_1,h_2\}$ defines an unique point in $Symm^2(\overline H_{\underline  n}^{un})$ and we are done.
\end{proof}

We now describe the image of the maps introduced in Theorem \ref{thm:geninj}. 
Note that these images lie in the boundary $\overline H_{k,g}-H_{k,g}$ of the corresponding Hurwitz scheme.
Let $\Delta^n_{k,g}$ be the closure of the locus in $\overline H_{k,g}$ of  admissible covers $\pi:C \rightarrow B$ where $B$ has $n+1$ irreducible components. 
Recall that $B$ is a stable pointed curve of genus $0$ having as marked points  the smooth branch locus of $\pi$.
It is a well known fact that $\Delta_{k,g}^n$ has pure codimension $n$ in $\overline H_{k,g}$ and $\Delta_{k,g}^n\subset \Delta_{k,g}^{n-1}$ (cf. \cite[page 181]{HM}). 
Furthermore, the irreducible components of $\Delta_{k,g}^n$ correspond to different stable arrangements  of the marked points in the components of $B$ . For instance, the components of $\Delta_{k,g}^1$ are the closures  $\Delta_{k,g}^{1,\ell}$ of the loci of those admissible covers where $B$ has two components, one of wich contains $\ell$ of the $b$ marked points, for $2\leq \ell\leq b/2$.

Let $\Delta_{k,g}^{2,(\ell_1,\ell_2,\ell_3)}$ be the closure of the locus of $k$-sheeted admissible covers $\pi:C\rightarrow B$ where $B$ has three irreducible components $B_1,B_2,B_3$ such that  $B_i$ meets $B_{i+1}$ in a single node, for $i=1,2$, $B_1$ and $B_3$ do not intersect, and each $B_i$ contains $\ell_i$ marked points, for $i=1,2,3$. 
Then the components of  $\Delta_{k,g}^2$ are those $\Delta_{k,g}^{2,(\ell_1,\ell_2,\ell_3)}$ where the arrangement of the marked points is stable, that is, where $\ell_2\geq1$ and $\ell_i\geq2$ for $i=1,3$.

\begin{theorem}
\label{thm:images}
The  maps defined in Theorem \ref{thm:geninj} satisfy
\[
\overline{\mathrm{Im}(\gamma)}=\Delta_{k+1,g}^{1,2},
\quad\quad
\overline{\mathrm{Im}(\phi)}=\Delta_{k+1,g+1}^{2,(2,2g+2k-2,2)},
\quad\quad
\overline{\mathrm{Im}(\psi)}=\Delta_{k,g+1}^{1,2}
\]
and
\[
\overline{\mathrm{Im}(\lambda)}=\Delta_{k_1+k_2-\delta,g_1+g_2+\delta-1 }^{1, \min\{2g_1+2k_1-2,2g_2+2k_2-2\}}.
\]
\end{theorem}
\begin{proof}
From the proofs of Theorems \ref{prop:step2}, \ref{irred:lem} and \ref{red:nonconjug0}, we see that 
\[{\mathrm{Im}(\gamma)}\subset\Delta_{k+1,g}^{1,2},
\quad\quad
\mathrm{Im}(\phi)\subset\Delta_{k+1,g+1}^{2,(2,2g+2k-2,2)},
\quad\quad
{\mathrm{Im}(\psi)}\subset\Delta_{k,g+1}^{1,2}\]
and
\[{\mathrm{Im}(\lambda)}\subset\Delta_{k_1+k_2-\delta,g_1+g_2+\delta-1 }^{1, \min\{2g_1+2k_1-2,2g_2+2k_2-2\}}.\]
As in the proof of Theorem \ref{thm:geninj}, we denote the maps $\gamma$, $\phi$, $\psi$ and $\lambda$ by $\nu$ and denote by $\overline H_{\underline n'}$ its target space. 
As we have seen above, the image of $\nu$ lies in an irreducible component $\Delta$ of $\Delta_{\underline n'}^{\varepsilon(\nu)}$, where $\varepsilon(\nu)=1$ for $\nu\in\{\lambda,\gamma,\psi\}$ and $\varepsilon(\phi)=2$.
%Since  $\overline H_{\underline n}$ is irreducible for any $\underline n$, this image is also irreducible. 
Since $\Delta$ is irreducible and has codimension $\varepsilon(\nu)$ in $\overline H_{\underline n'}$, to conclude that  $\overline{\mathrm{Im}(\nu)}= \Delta$ it is enough to notice that this image also has codimension $\varepsilon(\nu)$ in $\overline H_{\underline n'}$.
Since  $\nu$ is generically injective,
 this can be checked simply calculating the dimensions of target and source Hurwitz schemes of each map.
\end{proof}

\bibliographystyle{mn}

\addcontentsline{toc}{section}{References}

\begin{thebibliography}{999999}                                                                                             

\bibitem{B1} E. BALLICO, \textit{On the gonality of graph curves}, Manuscripta Math. \textbf{129,} N${{}^o}$ 2 (2009), 169-180.



\bibitem{Br} S. BRANNETTI, \textit{Aspects of Brill-Noether geometry in
moduli theory of algebraic and tropical curves}, Doctor thesis, Universit\`{a}
Degli Studi "Roma Tre", 2010/2011, available at $<$http://ricerca.mat.uniroma3.it/dottorato/Tesi/brannetti.pdf$>$.

\bibitem{C} L. CAPORASO, \textit{Gonality of algebraic curves and  graphs},
Algebraic and Complex Geometry: In honour of Klaus Hulek's 60th Birthday, Springer Proceedings in Mathematics and Statistics 71, pp 77-109 2014.



\bibitem{F} W. FULTON, \textit{Hurwitz schemes and irreducibility of moduli
of algebraic curves}, Ann. of Math (2) \textbf{90} (1969), 542-575.


\bibitem{HM} J. HARRIS and I. MORRISON, \textit{Moduli of curves}. Springer, New York, 1998.

\bibitem{HM82} J. HARRIS and D. MUMFORD, \textit{On the Kodaira dimension
of the moduli space of curves}, Invent. Math. \textbf{67} (1982), 23-86.



\bibitem{Hu} A. HURWITZ, \textit{Ueber Riemann'sche Fl\"{a}chen mit gegebenen Verzweigungspunkten.} Math. Ann. \textbf{39}, N${{}^o}$ 1 (1891), 1-60.

\bibitem{K2} F. KNUDSEN, \textit{The Projectivity of the moduli space of
stable curves, II: The stacks }$M_{g,n}$, Math. Scand. \textbf{52} (1983), 161-199.

\bibitem{K3} F. KNUDSEN, \textit{The Projectivity of the moduli space of stable curves, III:  The line bundles on $M_{g,n}$ and a proof of the projectivity of 
$\overline M_{g,n}$ in characteristic 0}, Math. Scand. \textbf{52} (1983), 200-212.



\bibitem{M} S. MOCHIZUKI, \textit{The Geometry of the Compactification of
the Hurwitz Scheme}, Publ. Res. Inst. Math. Sci \textbf{31}, N${{}^o}$ 3 (1995), 355-441.


\bibitem{S} F. SEVERI, \textit{Vorlesungen \"{u}ber algebraische  Geometrie:Geometrie auf einer Kurve, Riemannsche Fl\"{a}chen, Abelsche Integrale}. Bibliotheca Mathematica Teubneriana, Band 32. Leipzig, 1921, reprint, New York-London, Johnson, 1968.


\end{thebibliography}

\bigskip
\noindent{\smallsc Juliana Coelho, Universidade Federal Fluminense (UFF), 
Instituto de Matem\'atica e Estat\'istica - 
 Rua Prof. Marcos Waldemar de Freitas Reis, S/N, 
 Gragoatá, 24210-201 - Niterói -  RJ,  Brasil}\\
{\smallsl E-mail address: \small\verb?julianacoelhochaves@id.uff.br?}
\bigskip
\bigskip

\noindent{\smallsc Frederico Sercio, Universidade Federal de Juiz de Fora  (UFJF),  Departamento de Matem\'{a}tica, Rua Jos\'{e} Louren\c{c}o Kelmer, s/n - Campus Universit\'{a}rio, 36036-900 - Juiz de Fora - MG, Brasil}\\
{\smallsl E-mail address: \small\verb?fred.feitosa@ufjf.edu.br?}

\end{document}